\documentclass[12pt]{amsart}
\usepackage{amsfonts}

\newtheorem{theorem}{Theorem}[section]
\newtheorem{lemma}{Lemma}[section]
\newtheorem{proposition}{Proposition}[section]

\newtheorem{remark}{Remark}[section]
\newtheorem{definition}{Definition}[section]
\newtheorem{example}{Example}[section]

\def\phi{\varphi}

\email{marek.galewski@p.lodz.pl}
\email{gmolica@unirc.it}
\thanks{Typeset by \LaTeX}
\input{tcilatex}

\begin{document}
\title[Existence results...]{Existence results for one-dimensional\\
fractional equations}
\author{Marek Galewski}
\address[M. Galewski]{Institute of Mathematics, Technical University of
Lodz, Wolczanska 215, 90-924 Lodz, Poland}
\author{Giovanni Molica Bisci}
\address[G. Molica Bisci]{Dipartimento P.A.U., Universit\`a degli Studi
``Mediterranea" di Reggio Calabria, Salita Melissari - Feo di Vito, 89100
Reggio Calabria, Italy}

\begin{abstract}
In this note a critical point result for differentiable functionals is
exploited in order to prove that a suitable class of one-dimensional
fractional problems admits at least one non-trivial solution under an
asymptotical behaviour of the nonlinear datum at zero. A concrete example of
an application is then presented.
\end{abstract}

\maketitle


\section{Introduction}

\label{sec:introduzione}

Critical point theory has been very useful in determining the existence of
solutions for integer order differential equations with some boundary
conditions; see for instance, in the vast literature on the subject, the
classical books \cite{MW, Ra, struwe, wi} and references therein. But until
now, there are few results for fractional boundary value problems (briefly
BVP) which were established exploiting this approach, since it is often very
difficult to establish a suitable space and variational functional for
fractional problems. In the literature there are some approaches connected
with investigations of fractional boundary value problems with critical
point theory methods which depend on the type of fractional derivative used.
Although fractional calculus shares some common features with classicial
differential calculus, there are some obvious differences, for example in
the context of integration by parts, see for example \cite{s1}.

In this paper, overcoming the above mentioned difficulty, a new variational
approach is provided to investigate the existence of solutions to the
following fractional BVP, namely $(F_{f})$ and given by: 
\begin{equation*}
\begin{gathered} \frac{d}{dt} \Big({}_0 D_t^{\alpha-1}({}_0^c D_t^{\alpha}
u(t)) - {}_t D_T^{\alpha-1}({}_t^c D_T^{\alpha} u(t))\Big) + f(u(t)) = 0,
\,\text{a.e. } t \in [0, T] \\ u(0) = u(T) = 0, \end{gathered}
\end{equation*}%
where $\alpha \in (1/2,1]$, ${}_{0}D_{t}^{\alpha -1}$ and $%
{}_{t}D_{T}^{\alpha -1}$ are the left and right Riemann-Liouville fractional
integrals of order $1-\alpha $ respectively, $_{0}^{c}D_{t}^{\alpha }$ and $%
_{t}^{c}D_{T}^{\alpha }$ are the left and right Caputo fractional
derivatives of order $\alpha$ respectively, and $f:\mathbb{R}\rightarrow 
\mathbb{R}$ is a continuous function.

Fractional equations appear in concrete applications in many fields such as,
among the others, optimization, finance, stratified materials, conservation
laws, ultra-relativistic limits of quantum mechanics, minimal surfaces,
materials science and water waves. This is one of the reason why, recently,
non-local fractional problems are widely studied.

\indent An interesting physical case is briefly discussed in \cite{JZ} where
the authors are interested on the existence and multiplicity of solutions
for the following problem 
\begin{equation*}
\begin{gathered} \frac{d}{dt} \Big(D_\beta(u(t))\Big) + \nabla F(t,u(t)) =
0, \,\text{a.e. } t \in [0, T] \\ u(0) = u(T) = 0, \end{gathered}
\end{equation*}
where 
\begin{equation*}
D_\beta(u(t)):= \frac{1}{2}({}_0 D_t^{-\beta}(u^{\prime }(t)) +{}_t
D_T^{-\beta}(u^{\prime }(t))), 
\end{equation*}
$\beta\in [0,1)$, and $F:[0,T]\times \mathbb{R}^{N}\rightarrow \mathbb{R}$
(with $N\geq 1$) is a suitable given function. This model describes, as
stating, that the mass flux of a particle is related to the negative
gradient via a combination of the left and right fractional integrals.

\indent Engineering applications of fractional concepts are connected with
viscoelastic models, stochastic dynamics and with recently developed
fractional-order thermoelasticity \cite{[3]}. In these elds the main use of
fractional operators has been concerned with the interpolation between the
heat flux and its time-rate of change, that is related to the well-known
second sound e ect. In other recent studies \cite{[2]} a fractional,
non-local thermoelastic model has been proposed as a particular case of the
non-local, integral, thermoelasticity introduced in \cite{[1]}. We would
like to mention also work \cite{[4]} where the authors extend the non-local
model of fractional heat conduction to the case of of a purely elastic
material accounting for the thermoelastic coupling.

On the other hand, is whether or not the existence results got in the
classical context can be extended to the non-local framework of the
fractional Laplacian type operators. In this paper, motivated by a
theoretical point of view, we are interested on the one-dimensional setting,
previously considered in several papers (see, for instance, the manuscripts 
\cite{Bo, B03} and references therein).

More concretely, in Theorem~\ref{generalkernel0f} we prove the existence of
one solution to problem $(F_f)$ requiring a simple algebraic inequality
condition namely $(S_G)$; see Remarks \ref{Y1} and \ref{Y2}. A parametric
version of this result is successively discussed in Theorem \ref{secondo} in
which, for small values of the parameter and requiring an additional
asymptotical behaviour of the potential at zero if $f(0)=0$, the existence
of one non-trivial solution is achieved; see Remark \ref{atzero}.

Moreover, we deduce the existence of solutions for small positive values of
the parameter such that the corresponding solutions have smaller and smaller
energies as the parameter goes to zero; see, for more details, Remark \ref%
{behavious}.

The proof of Theorem~\ref{generalkernel0f} (as well as of Theorem \ref%
{secondo}) is based on variational techniques. Precisely, in the sequel we
will perform the variational principle of Ricceri obtained in \cite{Ricceri}%
. Moreover, for several related topics and a careful analysis of the
abstract framework we refer to the recent monograph~\cite{kristaly}.

A special case of our results reads as follows:

\begin{theorem}
\label{ab} Let $f:\mathbb{R}\rightarrow \mathbb{R}$ be a non-negative
continuous function and $\alpha \in (1/2, 1]$. Assume that 
\begin{equation*}
\lim_{\xi\rightarrow 0^+}\frac{f(t)}{t}=+\infty.\eqno{(S_0)}
\end{equation*}
Then, for every 
\begin{equation*}
\mu\in \Lambda:=\left(0,\frac{\Gamma(\alpha)^2 |\cos(\pi \alpha)|({2\alpha-1}%
)}{T^{2\alpha}}\left(\sup_{\gamma>0}\frac{\gamma^2}{\displaystyle %
\int_0^\gamma f(s)\,ds}\right)\right),
\end{equation*}
\noindent the following parametric problem 
\begin{equation*}
\begin{gathered} \frac{d}{dt} \Big({}_0 D_t^{\alpha-1}({}_0^c D_t^{\alpha}
u(t)) - {}_t D_T^{\alpha-1}({}_t^c D_T^{\alpha} u(t))\Big) + \mu f(u(t)) =
0, \,\text{a.e. } t \in [0, T] \\ u(0) = u(T) = 0, \end{gathered}
\end{equation*}

\noindent admits at least one non-trivial solution in $E^\alpha_0$.
Moreover, one has 
\begin{equation*}
\lim_{\mu\rightarrow 0^+}\int_0^T |_0^c D_t^{\alpha} u_\mu(t)|^2 dt=0,
\end{equation*}
and the function 
\begin{equation*}
\mu\mapsto - \int_0^T {}_0^c D_t^{\alpha} u_\mu(t) \cdot {}_t^c D_T^{\alpha}
u_\mu(t) dt-\mu \int_0^T \left(\int_0^{u_\mu(t)}f(s)ds\right)dt,
\end{equation*}
is negative and strictly decreasing in $\Lambda$.
\end{theorem}

We would like to emphasize that, as observed in Remark \ref{unb}, the energy
functional $J_{\mu }$ associated to the above parametric problem can be
unbounded from below in the ambient space $E_{0}^{\alpha }$. Hence, in order
to find critical points of $J_{\mu }$ we can not argue, in general, by
direct minimization techniques; see Example \ref{example} and Remark \ref%
{conclusion}.

Our assumptions also do not allow to use classical minimization topics and
related arguments.

Following \cite{JZ}, we point out that there are few results on the
solutions to fractional BVP which were established by the critical point
theory, since it is often very difficult to establish a suitable space and
variational functional for fractional differential equations with some
boundary conditions. These difficulties are mainly caused by the following
significative facts:

\begin{itemize}
\item[(i)] the composition rule in general fails to be satisfied by
fractional integral and fractional derivative operators;

\item[(ii)] the fractional integral is a singular integral operator and
fractional derivative operator is non-local;

\item[(iii)] the adjoint of a fractional differential operator is not the
negative of itself.
\end{itemize}

\indent It should be mentioned here that the fractional variational
principles were started to be investigated deeply. The fractional calculus
of variations was introduced by Riewe in \cite{Ri} where he presented a new
approach to mechanics that allows one to obtain the equations for a
nonconservative system using certain functionals.

For completeness, we recall that a careful and interesting analysis of the
elliptic fractional case was developed in the recent and nice works \cite%
{sY,sBNRES, servadeiKavian, servadeivaldinociREGO} and references therein.

The paper is organized as follows. In Section~\ref{sec:preliminary} we give
the principal definitions related to our abstract functional framework. In
Section~\ref{sec:main} we prove Theorems~\ref{generalkernel0f} and~\ref%
{secondo}, while Section~\ref{sec:comm} is devoted to some comments on the
results of the paper. Finally, an application of Theorem~\ref{secondo}, is
presented in Example \ref{Y2} studying a one-dimensional fractional equation
involving a suitable non-linearity.

\section{Some preliminaries}

\label{sec:preliminary} This section is devoted to the notations used along
the paper. We also give some preliminary results which will be useful in the
sequel.

\subsection{The functional setting}

\label{subsec:X0}

\begin{definition}
\label{def2.1} \textrm{Let $u$ be a function defined on $[a , b]$. The left
and right Riemann-Liouville fractional integrals of order $\alpha>0$ for a
function $u$ are defined by 
\begin{equation*}
{}_a D_t^{-\alpha}u(t) := \frac{1}{\Gamma(\alpha)} \int_a^t (t-s)^{\alpha-1}
u(s) ds,\newline
\end{equation*}
and 
\begin{equation*}
{}_t D_b^{-\alpha}u(t) := \frac{1}{\Gamma(\alpha)} \int_t^b (s-t)^{\alpha-1}
u(s) ds,
\end{equation*}
for every $t \in [a, b]$, provided the right-hand sides are pointwise
defined on $[a, b]$, where $\Gamma(\alpha)$ is the standard gamma function
given by 
\begin{equation*}
\Gamma(\alpha):=\int_0^{+\infty}z^{\alpha-1}e^{-z}dz.
\end{equation*}
}
\end{definition}

Set $AC^n([a, b], {\mathbb{R}})$ the space of functions $u:[a,b]\rightarrow {%
\mathbb{R}}$ such that $u\in C^{n-1}([a, b], {\mathbb{R}})$ and $%
u^{(n-1)}\in AC([a, b], {\mathbb{R}})$. Here, as usual, $C^{n-1}([a, b], {%
\mathbb{R}})$ denotes the set of mappings having $(n-1)$ times continuously
differentiable on $[a, b]$. In particular we denote $AC([a, b], {\mathbb{R}}%
):=AC^1([a, b], {\mathbb{R}})$.

\begin{definition}
\label{def2.2} \textrm{\ Let $\gamma \ge 0$ and $n \in \mathbb{N}$. }

\textrm{$(i)$ If $\gamma \in (n-1, n)$ and $u\in AC^n([a, b], {\mathbb{R}})$%
, then the left and right Caputo fractional derivatives of order $\gamma$
for function $u$ denoted by $_a^c D_t^{\gamma}u(t)$ and $_t^c
D_b^{\gamma}u(t)$, respectively, exist almost everywhere on $[a, b]$, $_a^c
D_t^{\gamma}u(t)$ and $_t^c D_b^{\gamma}u(t)$ are represented by 
\begin{equation*}
{}_a^c D_t^{\gamma}u(t) = \frac{1}{\Gamma(n-\gamma)} \int_a^t
(t-s)^{n-\gamma-1} u^{(n)}(s) ds,
\end{equation*}
and 
\begin{equation*}
{}_t^c D_b^{\gamma}u(t) = \frac{(-1)^n}{\Gamma(n-\gamma)} \int_t^b
(s-t)^{n-\gamma-1} u^{(n)}(s) ds,
\end{equation*}
for every $t \in [a, b]$, respectively. }

\textrm{$(ii)$ If $\gamma = n - 1$ and $u \in AC^{n-1}([a, b], {\mathbb{R}})$%
, then $_a^c D_t^{n-1}u(t)$ and $_t^c D_b^{n-1}u(t)$ are represented by 
\begin{equation*}
{}_a^c D_t^{n-1}u(t) = u^{(n-1)}(t), \quad\text{and}\quad _t^c D_b^{n-1}u(t)
= (-1)^{(n-1)} u^{(n-1)}(t),
\end{equation*}
}
\end{definition}

\noindent for every $t \in [a, b]$.

With these definitions, we have the rule for fractional integration by
parts, and the composition of the Riemann-Liouville fractional integration
operator with the Caputo fractional differentiation operator, which were
proved in \cite{k1} and \cite{s1}.

\begin{proposition}
\label{prop2.1} We have the following property of fractional integration 
\begin{equation}
\int_a^b [_a D_t^{-\gamma}u(t)] v(t) dt = \int_a^b [_t D_b^{-\gamma}v(t)]
u(t) dt, \quad \gamma > 0,  \label{e2.1}
\end{equation}
provided that $u \in L^p([a, b], {\mathbb{R}})$, $v \in L^q([a, b], {\mathbb{%
R}})$ and $p \ge 1$, $q \ge 1$, $1/p + 1/q \le 1 + \gamma$ or $p \neq 1$, $q
\neq 1$, $1/p + 1/q = 1 + \gamma$.
\end{proposition}

\begin{proposition}
\label{prop2.2 } Let $n \in \mathbb{N}$ and $n-1 < \gamma \le n$. If $u \in
AC^n([a, b], {\mathbb{R}})$ or $u \in C^n([a, b], {\mathbb{R}})$, then 
\begin{equation*}
{}_a D_t^{-\gamma}({}_a^c D_t^{\gamma} u(t)) = u(t) - \sum _{j=0}^{n-1} 
\frac{u^{(j)}(a)}{j!}(t-a)^j,
\end{equation*}
and 
\begin{equation*}
{}_t D_b^{-\gamma}({}_t^c D_b^{\gamma} u(t)) = u(t) - \sum _{j=0}^{n-1} 
\frac{(-1)^ju^{(j)}(b)}{j!}(b-t)^j,
\end{equation*}
for every $t \in [a, b]$. In particular, if $0 < \gamma \le 1$ and $u \in
AC([a, b], {\mathbb{R}})$ or $u \in C^1([a, b], {\mathbb{R}})$, then 
\begin{equation*}
{}_a D_t^{-\gamma}({}_a^c D_t^{\gamma} u(t)) = u(t) - u(a),
\end{equation*}
{and} 
\begin{equation*}
_t D_b^{-\gamma}({}_t^c D_b^{\gamma} u(t)) = u(t) - u(b).  \label{e2.2}
\end{equation*}
\end{proposition}

\begin{remark}
\label{rmk2.1} \textrm{We recall that a function $u\in AC([0,T],\mathbb{R}) $
is said to be a solution of $(F_{f})$ if the map 
\begin{equation*}
t\mapsto {}_0 D_t^{\alpha-1}({}_0^c D_t^{\alpha} u(t)) - {}_t
D_T^{\alpha-1}({}_t^c D_T^{\alpha} u(t)),
\end{equation*}
is derivable $($in the classical sense$)$ for almost every $t\in [0,T]$, and 
\begin{equation*}
\begin{gathered} \frac{d}{dt} \Big({}_0 D_t^{\alpha-1}({}_0^c D_t^{\alpha}
u(t)) - {}_t D_T^{\alpha-1}({}_t^c D_T^{\alpha} u(t))\Big) + f(u(t)) = 0,
\,\text{a.e. } t \in [0, T] \\ u(0) = u(T) = 0. \end{gathered}
\end{equation*}%
}
\end{remark}

To establish a variational structure for $(F_f)$, it is necessary to
construct appropriate function spaces.

Following \cite{JZ}, denote by $C_0^{\infty}([0, T],\mathbb{R})$ the set of
all functions $g \in C^{\infty}([0, T],\mathbb{R})$ with $g(0) = g(T) = 0$.

\begin{definition}
\label{def2.3} \textrm{\ Let $0 < \alpha \le 1$. The fractional derivative
space $E_0^{\alpha}$ is defined by the closure of $C_0^{\infty}([0, T],%
\mathbb{R})$ with respect to the norm 
\begin{equation*}
\|u\| := \Big(\int_0^T |_0^c D_t^{\alpha} u(t)|^2 dt + \int_0^T |u(t)|^2 dt %
\Big)^{1/2}.
\end{equation*}
}
\end{definition}

As observed in \cite{JZ}, the space $E_0^{\alpha}$ is a Hilbert space with
norm 
\begin{equation*}
\|u\|_{\alpha} := \Big(\int_0^T |_0^c D_t^{\alpha} u(t)|^2 dt\Big)^{1/2}.
\end{equation*}
\noindent For every $u\in E_{0}^{\alpha }$, set 
\begin{equation*}
\Vert u\Vert _{L^{s}}:=\Big(\int_{0}^{T}|u(t)|^{s}dt\Big)^{1/s},\,\,\,\,\,%
\,(s\geq 1)
\end{equation*}%
and 
\begin{equation*}
\Vert u\Vert _{\infty }:=\max_{t\in \lbrack 0,T]}|u(t)|.
\label{clas-embending}
\end{equation*}

\indent The next result will be crucial in the sequel.

\begin{lemma}
\label{EM} Assume $\alpha \in (1/2, 1]$ and let $\{u_j\}\subset
E_{0}^{\alpha}$ be a sequence weakly convergent to $u\in E_{0}^{\alpha}$.
Then, $u_j\rightarrow u$ in $C^0([0,T],\mathbb{R})$, i.e. $%
\|u_j-u\|_\infty\rightarrow 0,$ as $j\rightarrow \infty$.
\end{lemma}

See \cite[Proposition 3.3]{JZ}.

\smallskip \noindent Finally, one has the following two Lemmas.

\begin{lemma}
\label{lem2.1} Let $\alpha \in (1/2, 1]$. For every $u\in E_0^{\alpha}$, we
have 
\begin{gather*}
\|u\|_{L^2} \le \frac{T^{\alpha}}{\Gamma(\alpha+1)} \|\empty_0^c
D_t^{\alpha}u\|_{L^2},  \label{e2.4} \\
\|u\|_{\infty} \le \frac{T^{\alpha - \frac{1}{2}}}{\Gamma(\alpha)\sqrt{%
2\alpha-1}} \|_0^c D_t^{\alpha} u\|_{L^2}.  \label{e2.5}
\end{gather*}
\end{lemma}

\begin{lemma}
\label{lem2.2} Let $\alpha \in (1/2, 1]$, then for every $u \in E_0^{\alpha} 
$, we have 
\begin{equation*}
|\cos (\pi \alpha)| \|u\|_{\alpha}^2 \le - \int_0^T {}_0^c D_t^{\alpha} u(t)
\cdot {}_t^c D_T^{\alpha} u(t) dt \le \frac{1}{|\cos (\pi \alpha)|}
\|u\|_{\alpha}^2.  \label{e2.7}
\end{equation*}
\end{lemma}

See \cite{JZ} for details.

\subsection{A critical points result for differentiable functionals}

\label{subsec:ricceri} In order to prove our main result, stated in Theorem~%
\ref{generalkernel0f}, in the following we will perform the variational
principle of Ricceri established in \cite{Ricceri}. For the sake of clarity,
we recall it here below in the form given in \cite{BMB}.

\begin{theorem}
\label{BMB} Let $Y$ be a reflexive real Banach space, and $\Phi,\Psi:Y\to%
\mathbb{R}$ be two G\^{a}teaux differentiable functionals such that $\Phi$
is strongly continuous, sequentially weakly lower semicontinuous and
coercive in $Y$ and $\Psi$ is sequentially weakly upper semicontinuous in $Y$%
. Let $J_\mu$ be the functional defined as $J_\mu:=\Phi-\mu\Psi$, $\mu\in 
\mathbb{R}$\,, and for any ${\displaystyle r>\inf_Y\Phi}$ let $\varphi$ be
the function defined as 
\begin{equation*}
\varphi(r):=\inf_{u\in\Phi^{-1}\big((-\infty,r)\big)} \frac{\displaystyle%
\sup_{v\in\Phi^{-1}\big((-\infty,r)\big)}\Psi(v)-\Psi(u)}{r-\Phi(u)}\,\,.
\end{equation*}

\noindent Then, for any ${\displaystyle r>\inf_Y\Phi}$ and any $\mu\in
(0,1/\varphi(r))$\footnote{%
Note that, by definition, $\varphi(r)\geq 0$ for any ${\displaystyle %
r>\inf_Y\Phi}$\,. Here and in the following, if $\varphi(r)=0$, by $%
1/\varphi(r)$ we mean $+\infty$, i.e. we set $1/\varphi(r)=+\infty$\,.}, the
restriction of the functional $J_\mu$ to $\Phi^{-1}\big((-\infty,r)\big)$
admits a global minimum, which is a critical point $($precisely a local
minimum$)$ of $J_\mu$ in $Y$.
\end{theorem}

\section{The Main Results}

\label{sec:main} This section is devoted to the proof of the main result of
the present paper, that is the following.

\begin{theorem}
\label{generalkernel0f} Let $f:\mathbb{R}\rightarrow \mathbb{R}$ be a
continuous function and $\alpha \in (1/2, 1]$. Set 
\begin{equation*}
\kappa_\alpha:=\frac{T^{2\alpha}}{\Gamma(\alpha)^2 |\cos(\pi \alpha)|({%
2\alpha-1})},
\end{equation*}
where $\Gamma$ is the Euler function. Assume that 
\begin{equation*}
\sup_{\gamma>0}\frac{\gamma^2}{\displaystyle \max_{|\xi|\leq
\gamma}\int_0^\xi f(t)dt}>\kappa_\alpha.\eqno{(S_G)}
\end{equation*}
Then problem $(F_f)$ admits at least one solution in $E^\alpha_0$.
\end{theorem}

\begin{proof}
The idea of the proof consists in applying \cite[Theorem~2.1; part~$a)$]{BMB}%
) taking $Y:=E_0^{\alpha}$.

Hence, for given $u \in E_0^{\alpha}$, we define functionals $\Phi, \Psi :
E^{\alpha}_0 \to \mathbb{R}$ as follows: 
\begin{equation*}
\Phi(u) := - \int_0^T {}_0^c D_t^{\alpha} u(t) \cdot {}_t^c D_T^{\alpha}
u(t) dt, \,\,\,\,\,\mathrm{and}\,\,\,\,\, \Psi(u) := \int_0^T F(u(t))dt,
\end{equation*}
where $F(\xi):=\displaystyle \int_0^\xi f(s)ds,$ for every $\xi\in\mathbb{R}$%
.

\indent Clearly, $\Phi $ and $\Psi $ are continuously G\^{a}teaux
differentiable functional whose G\^{a}teaux derivatives at the point $u\in
E_{0}^{\alpha }$ are given by 
\begin{gather*}
\Phi ^{\prime }(u)(v)=-\int_{0}^{T}({}_{0}^{c}D_{t}^{\alpha }u(t)\cdot
{}_{t}^{c}D_{T}^{\alpha }v(t)+{}_{t}^{c}D_{T}^{\alpha }u(t)\cdot
{}_{0}^{c}D_{t}^{\alpha }v(t))dt, \\
\Psi ^{\prime
}(u)(v)=\int_{0}^{T}f(u(t))v(t)dt=-\int_{0}^{T}\int_{0}^{t}f(u(s))ds\cdot
v^{\prime }(t)dt,
\end{gather*}%
for every $v\in E_{0}^{\alpha }$.

Moreover, it is easy to see that 
\begin{equation*}
\Phi ^{\prime }(u)(v)=\int_{0}^{T}({}_{0}D_{t}^{\alpha
-1}({}_{0}^{c}D_{t}^{\alpha }u(t))-{}_{t}D_{T}^{\alpha
-1}({}_{t}^{c}D_{T}^{\alpha }u(t)))\cdot v^{\prime }(t)dt.
\end{equation*}%
Thus, the functional $J:=\Phi -\Psi \in C^{1}(E_{0}^{\alpha },\mathbb{R})$
and the functionals $\Phi $ and $\Psi $ are respectively sequentially weakly
lower and upper semicontinuous. As concerns functional $\Phi $, this follows
by Lemma \ref{EM}. Indeed $\Phi$ is strongly continuous and convex and hence
sequentially weakly lower semicontinuous. Moreover, the weakly upper
semicontinuity the functional $\Psi$ can be proved arguing in a standard way
by using again the compact embedding $E_{0}^{\alpha }\hookrightarrow
C^0([0,T],\mathbb{R})$.

Further, from Lemma \ref{lem2.2}, the functional $\Phi$ is coercive. Indeed,
one has 
\begin{equation*}
\Phi(u):= - \int_0^T {}_0^c D_t^{\alpha} u(t) \cdot {}_t^c D_T^{\alpha} u(t)
dt \geq|\cos (\pi \alpha)| \|u\|_{\alpha}^2\rightarrow +\infty,
\end{equation*}
as $\|u\|_{\alpha}\rightarrow +\infty$.

Moreover, a critical point of the functional $J$ is a solution of $(F_f)$.

\noindent Indeed, if $u_* \in E_0^{\alpha}$ is a critical point of $J,$ then 
\begin{equation}
\begin{aligned} 0 = J'(u_*)(v) &=\int_0^T \Big({}_0 D_t^{\alpha-1}({}_0^c
D_t^{\alpha} u_*(t)) - {}_t D_T^{\alpha-1}({}_t^c D_T^{\alpha} u_*(t)) \\
&\quad + \int_0^t f(u_*(s))ds\Big) v'(t) dt, \end{aligned}  \label{e3.1}
\end{equation}
for every $v \in E_0^{\alpha}$.

Now, we can choose $v \in E_0^{\alpha}$ such that 
\begin{equation*}
v(t) := \sin \frac{2k\pi t}{T} \quad \text{ or} \quad v(t) := 1 - \cos \frac{%
2k\pi t}{T}, \quad (k = 1, 2, \dots).
\end{equation*}
The theory of Fourier series and \eqref{e3.1} imply 
\begin{equation}
{}_0 D_t^{\alpha-1}({}_0^c D_t^{\alpha} u_*(t)) - {}_t D_T^{\alpha-1}({}_t^c
D_T^{\alpha} u_*(t)) + \int_0^t f(u_*(s))ds = \kappa  \label{e3.2}
\end{equation}
a.e. on $[0, T]$ for some $\kappa \in \mathbb{R}$.

By \eqref{e3.2}, it is easy to show that $u_* \in E_0^{\alpha}$ is a
solution of $(F_f)$. Now, we look on the existence of a critical point of
the functional $J$ in $E_0^{\alpha}$.

Since condition $(S_G)$ holds, there exists $\bar \gamma>0$ such that 
\begin{equation}  \label{ser1}
\frac{\bar\gamma^2}{\displaystyle \max_{|\xi|\leq \bar\gamma}\int_0^\xi
f(t)dt}>\frac{T^{2\alpha}}{\Gamma(\alpha)^2 |\cos(\pi \alpha)|({2\alpha-1})}.
\end{equation}
\indent Now, by Lemma \ref{lem2.1} (note that $\alpha > 1/2$), for each $u
\in E_0^{\alpha}$ we have 
\begin{equation}
\|u\|_{\infty} \le c\Big(\int_0^T |_0^c D_t^{\alpha} u(t)|^2 dt \Big)^{1/2}
= c \|u\|_{\alpha},  \label{e3.3}
\end{equation}
where 
\begin{equation*}
c:= \frac{T^{\alpha-\frac{1}{2}}}{\Gamma(\alpha) \sqrt{2\alpha-1}}.
\end{equation*}
\indent Hence, set 
\begin{equation}  \label{e3.7}
\displaystyle r:=\frac{|\cos(\pi \alpha)|}{c^2}\bar \gamma^2.
\end{equation}
\indent Moreover, for every $u \in E_0^{\alpha}$ such that $u \in
\Phi^{-1}((-\infty, r))$, by Lemma \ref{e2.7} we have 
\begin{equation*}
|\cos (\pi \alpha)| \|u\|_{\alpha}^2 \le \Phi(u) < r,
\end{equation*}
which implies 
\begin{equation}
\|u\|_{\alpha}^2 <\frac{r}{|\cos(\pi \alpha)|} .  \label{e3.8}
\end{equation}
\indent Thus, by \eqref{e3.3}, \eqref{e3.7} and \eqref{e3.8} we obtain 
\begin{equation*}
|u(t)| \leq c \|u\|_{\alpha} < c \sqrt{\frac{r}{|\cos(\pi \alpha)|}} = \bar
\gamma, \quad \forall\; t \in [0, T].
\end{equation*}
\indent Hence, 
\begin{equation*}
\Psi(u) = \int_0^T F(u(t)) dt \le \int_0^T \max _{|\xi| \le \bar \gamma}
F(\xi) dt = T \max _{|\xi| \le \bar \gamma} F(\xi),
\end{equation*}
for every $u \in E_0^{\alpha}$ such that $u \in \Phi^{-1}((-\infty, r))$.

Then, 
\begin{equation*}
\sup _{u \in \Phi^{-1}((-\infty, r))} \Psi(u) \le T \max _{|\xi| \le \bar
\gamma} F(\xi).
\end{equation*}

Taking into account the above computations and remarks, one has the
following inequalities 
\begin{eqnarray*}
\varphi(r) &=& \inf_{u \in \Phi^{-1}((-\infty, r))}\frac{\displaystyle \sup
_{v \in \Phi^{-1}((-\infty, r))} \Psi(v)- \Psi(u)}{r-\Phi(u)}  \notag \\
&\leq & \frac{\displaystyle \sup _{v \in \Phi^{-1}((-\infty, r))} \Psi(v)}{r}
\\
&\leq& \frac{c^2T}{|\cos(\pi \alpha)|} \frac{\displaystyle\max_{|\xi|\leq
\bar \gamma}F(\xi)}{\bar \gamma^2}.
\end{eqnarray*}

Thus, it follows that 
\begin{equation}  \label{ser2}
\varphi(r)\leq \kappa_\alpha \frac{\displaystyle\max_{|\xi|\leq \bar
\gamma}F(\xi)}{\bar \gamma^2},
\end{equation}
observing that 
\begin{equation*}
\kappa_\alpha=\frac{c^2T}{|\cos(\pi \alpha)|}.
\end{equation*}
\indent Consequently, by \eqref{ser1} and \eqref{ser2} one has $\varphi(r)<1$%
. Hence, since $1\in\left(0,{1}/{\varphi(r)}\right)$, Theorem \ref{BMB}
ensures that the functional $J$ admits at least one critical point (local
minima) $\widetilde{u}\in \Phi^{-1}((-\infty,r))$. The proof is complete.
\end{proof}

We explicitly note that Theorem \ref{BMB} can be exploited proving the
existence of one solution for the parametric version of problem $(F_f)$,
namely $(F_f^{\mu})$ and given by: 
\begin{equation*}
\begin{gathered} \frac{d}{dt} \Big({}_0 D_t^{\alpha-1}({}_0^c D_t^{\alpha}
u(t)) - {}_t D_T^{\alpha-1}({}_t^c D_T^{\alpha} u(t))\Big) + \mu f(u(t)) =
0, \,\text{a.e. } t \in [0, T] \\ u(0) = u(T) = 0. \end{gathered}
\end{equation*}
\indent More precisely, one has the following existence property.

\begin{theorem}
\label{secondo} Let $f:\mathbb{R}\rightarrow \mathbb{R}$ be a continuous
function and $\alpha \in (1/2, 1]$. Then, for every $\mu$ sufficiently
small, i.e. 
\begin{equation*}
\mu\in\left(0,\frac{1}{\kappa_\alpha}\left(\sup_{\gamma>0}\frac{\gamma^2}{%
\displaystyle\max_{|\xi|\leq \gamma}F(\xi)}\right)\right),
\end{equation*}
problem $(F_f^{\mu})$ admits at least one solution $u_\mu\in E^\alpha_0$.
\end{theorem}

\begin{proof}

Let us take 
\begin{equation*}
0<\mu<\frac{1}{\kappa_\alpha}\left(\sup_{\gamma>0}\frac{\gamma^2}{%
\displaystyle\max_{|\xi|\leq \gamma}F(\xi)}\right).
\end{equation*}
Hence, there exists $\bar\gamma>0$ such that 
\begin{equation}  \label{ser12345}
\kappa_\alpha\mu<\frac{\bar\gamma^2}{\displaystyle \max_{|\xi|\leq
\bar\gamma}F(\xi)}.
\end{equation}
Set 
\begin{equation}  \label{e3.7345}
\displaystyle r:=\frac{|\cos(\pi \alpha)|}{c^2}\bar \gamma^2.
\end{equation}
Preserving the notations as in the proof of Theorem \ref{generalkernel0f},
one has 
\begin{equation*}
\varphi(r)\leq\frac{\displaystyle \sup _{v \in \Phi^{-1}((-\infty, r))}
\Psi(v)}{r}\leq \frac{c^2T}{|\cos(\pi \alpha)|} \frac{\displaystyle%
\max_{|\xi|\leq \bar \gamma}F(\xi)}{\bar \gamma^2}<\frac{1}{\mu}.
\end{equation*}
Hence, since $\mu\in\left(0,{1}/{\varphi(r)}\right)$, Theorem \ref{BMB}
ensures that the functional $J_\mu$ admits at least one critical point
(local minima) $u_\mu\in \Phi^{-1}((-\infty,r))$. The proof is complete.
\end{proof}

\begin{remark}
\label{unb}\textrm{{\ Our variational approach consist in looking for
critical points of the functional~$J_{\mu}$ naturally associated with
problem $(F_f^{\mu})$. We would like to note that, in general, $J_{\mu}$ can
be unbounded from below in $E_0^\alpha$.}}

\textrm{Indeed, for instance, in the case when $f(t):=1+|t|^{q-2}t$ with $%
q\in (2, +\infty)$, for any fixed $u\in E_0^\alpha\setminus\{0\}$ and $%
\tau\in\mathbb{R}$, we get 
\begin{equation*}
\begin{aligned} J_{\mu}(\tau u) & = \Phi(\tau u)-\mu\int_0^{T} F(\tau
u(t))\,dt\\ & \leq \frac{\tau^2}{|\cos (\pi \alpha)|} \|u\|_{\alpha}^2-\mu
\tau\|u\|_{L^1}-\frac{\mu \tau^q}{q}\|u\|_{L^q}^q \to -\infty \end{aligned}
\end{equation*}
as $\tau\to +\infty$.}

\textrm{Hence, in order to find critical points of $J_{\mu}$ we can not
argue, in general, by direct minimization. }
\end{remark}

\section{Some Comments}

\label{sec:comm} In this section we give some remarks and a concrete example
of application of our results.

\begin{remark}
\label{Y1}\textrm{{If in Theorem \ref{generalkernel0f} the function $f$ is
non-negative, condition $(S_G)$ assumes the more simple and significative
form 
\begin{equation*}
\sup_{\gamma>0}\frac{\gamma^2}{\displaystyle \int_0^\gamma f(s)ds}%
>\kappa_\alpha.\eqno{(S_G')}
\end{equation*}
Moreover, if the following assumption is verified 
\begin{equation*}
\limsup_{\xi\rightarrow +\infty}\frac{\xi^2}{\displaystyle \int_0^\xi f(s)ds}%
>\kappa_\alpha,\eqno{(S_\infty)}
\end{equation*}
then, condition $(S_G^{\prime })$ automatically holds. } }
\end{remark}

\begin{remark}
\label{Y2}\textrm{{Let $\bar \gamma>0$ be a real constant such that 
\begin{equation*}
\frac{{\bar\gamma}^2}{\displaystyle \max_{|\xi|\leq \bar \gamma}\int_0^{\xi}
f(s)ds}>\kappa_\alpha,
\end{equation*}
and said $\widetilde{u}\in E^\alpha_0$ be the solution of problem $(F_f)$
obtained by using Theorem \ref{BMB}. Hence, since $\widetilde{u}%
\in\Phi^{-1}((-\infty,r))$, it follows that $\|\widetilde{u}\|_\infty\leq
\bar \gamma. $ } }
\end{remark}

\begin{remark}
\label{atzero}\textrm{{If in Theorem \ref{secondo} one has $f(0)\neq 0$,
then the obtained solution is clearly non-trivial. On the other hand, the
non-triviality of the solution can be achieved also in the case $f(0)=0$
requiring the additional condition at zero 
\begin{equation}  \label{ZeRo}
\lim_{t\rightarrow 0^+}\frac{F(t)}{t^2}=+\infty.
\end{equation}
}}

\textrm{\indent Indeed, let $0<\bar\mu<\mu^{\star}$. Then, there exists $%
\bar\gamma>0$ such that 
\begin{equation}  \label{ser123}
\kappa_\alpha\bar\mu<\frac{\bar\gamma^2}{\displaystyle \max_{|\xi|\leq
\bar\gamma}F(\xi)}.
\end{equation}
}

\textrm{Thanks to Theorem \ref{BMB}, for every $\mu\in (0,\bar\mu)$ there
exists a critical point of $J_\mu$ such that 
\begin{equation*}
u_\mu\in\Phi^{-1}((-\infty,r_{\bar\mu})),
\end{equation*}
where 
\begin{equation*}
\displaystyle r_{\bar\mu}:=\frac{|\cos(\pi \alpha)|}{c^2}\bar \gamma^2,
\end{equation*}
\indent In particular, $u_\mu$ is a global minimum of the restriction of $%
J_\mu$ to $\Phi^{-1}((-\infty,r_{\bar\mu}))$.}

\textrm{We will prove that the function $u_\mu$ cannot be trivial. To this
end, let us show that 
\begin{equation}  \label{inf}
\limsup_{\|u\|_\alpha\rightarrow 0^+}\frac{\Psi(u)}{\Phi(u)}=+\infty.
\end{equation}
\indent Due to our assumptions at zero, we can fix a sequence $%
\{\xi_j\}\subset \mathbb{R}^{+}$ converging to zero and two constant $\sigma$%
, and $\kappa$ $($with $\sigma>0$$)$ such that 
\begin{equation*}
\lim_{j\rightarrow \infty}\frac{\displaystyle F(\xi_j)}{\xi_j^2}=+\infty,
\end{equation*}
and 
\begin{equation*}
F(\xi)\geq \kappa \xi^2,
\end{equation*}
\noindent for every $\xi\in [0,\sigma]$.}

\textrm{\noindent Now, fix a function $v\in E^\alpha_0$ $($note that $%
C^{\infty}[0, T]\subset E^\alpha_0$$)$ such that:}

\begin{itemize}
\item[i)] \textrm{$v(t)\in [0,1]$, for every $t\in [0,T]$; }

\item[ii)] \textrm{$v(t)=1$, for every $t\in [\frac{T}{4},\frac{3T}{4}]$. }
\end{itemize}

\textrm{Hence, fix $M>0$ and consider a real positive number $\eta$ with }

\textrm{%
\begin{equation*}
M<\frac{\displaystyle \eta T+\kappa\int_{[0,T]\setminus [\frac{T}{4},\frac{3T%
}{4}]}v(t)^2dt}{\Phi(v)}.
\end{equation*}
}

\textrm{\indent Then, there is $\nu\in\mathbb{N}$ such that $\xi_j<\sigma$
and 
\begin{equation*}
\int_0^{\xi_j}f(s)ds\geq 2\eta\xi_j^2,
\end{equation*}
\noindent for every $j>\nu$.}

\textrm{\indent At this point, for every $j>\nu$, and bearing in mind the
properties of the function $v$ $($$0\leq\xi_jv(t)<\sigma$ for $j$
sufficiently large$)$, one has 
\begin{eqnarray*}
\frac{\Psi(\xi_j v)}{\Phi(\xi_j v)} &=& \frac{\displaystyle \int_{[\frac{T}{4%
},\frac{3T}{4}]}\left(\displaystyle\int_0^{\xi_j}f(s)ds\right)dt+\int_{[0,T]%
\setminus [\frac{T}{4},\frac{3T}{4}]}F(\xi_jv(t))dt}{\Phi(\xi_j v)}  \notag
\\
&\geq& \frac{\displaystyle \eta T+\kappa\int_{[0,T]\setminus [\frac{T}{4},%
\frac{3T}{4}]}v(t)^2dt}{\Phi(v)}>M.
\end{eqnarray*}
\indent Since $M$ could be arbitrarily large, it follows that 
\begin{equation*}
\lim_{j\rightarrow \infty}\frac{\Psi(\xi_j v)}{\Phi(\xi_j v)}=+\infty,
\end{equation*}
\noindent from which \eqref{inf} clearly follows.}

\textrm{Hence, there exists a sequence $\{w_j\}\subset E^\alpha_0$ strongly
converging to zero, such that, for every $j$ sufficiently large, $w_j\in
\Phi^{-1}((-\infty,r_{\bar\mu}))$, and 
\begin{equation*}
J_\mu(w_j):=\Phi(w_j)-\mu\Psi(w_j)<0.
\end{equation*}
\indent Since $u_\mu$ is a global minimum of the restriction of $J_\mu$ to $%
\Phi^{-1}((-\infty,r_{\bar\mu}))$, we conclude that 
\begin{equation}  \label{ne}
J_\mu(u_\mu)<0,
\end{equation}
so that $u_\mu$ is not trivial. }
\end{remark}

\begin{remark}
\label{behavious}\textrm{{\ Put 
\begin{equation*}
\mu^{\star}:=\frac{1}{\kappa_\alpha}\left(\sup_{\gamma>0}\frac{\gamma^2}{%
\displaystyle\max_{|\xi|\leq \gamma}F(\xi)}\right).
\end{equation*}
\indent From \eqref{ne} we easily see that the map 
\begin{equation}  \label{mapnegative}
(0, \mu^\star)\ni \mu\mapsto J_{\mu}(u_\mu)\,\,\, \mbox{is negative}.
\end{equation}
\noindent Further, we claim that 
\begin{equation*}
\displaystyle \lim_{\mu\rightarrow 0^+}{\|u_{\mu}\|_{\alpha}}=0.
\end{equation*}
}}

\textrm{Indeed, let $0<\bar\mu<\mu^{\star}$. Thus, there exists $%
\bar\gamma>0 $ such that 
\begin{equation}  \label{ser1234}
\kappa_\alpha\bar\mu<\frac{\bar\gamma^2}{\displaystyle \max_{|\xi|\leq
\bar\gamma}F(\xi)}.
\end{equation}
}

\textrm{Set 
\begin{equation*}
\displaystyle r_{\bar \mu}:=\frac{|\cos(\pi \alpha)|}{c^2}\bar \gamma^2.
\end{equation*}
Bearing in mind that $\Phi$ is coercive and that for every $\mu\in
(0,\bar\mu]$ the solution $u_\mu\in \Phi^{-1}((-\infty,r_{\bar \mu}))$, one
has that there exists a positive constant $L$ such that 
\begin{equation*}
\|u_\mu\|_{\alpha}\leq L,
\end{equation*}
for every $\mu\in (0,\bar\mu]$.}

\textrm{Then, one has that 
\begin{equation}  \label{bounded}
\begin{aligned} \left|\int_0^T f(u_{\mu}(t))u_{\mu}(t)dt\right| & \leq
cL\max_{|s|\leq cL}|f(s)|T, \end{aligned}
\end{equation}
\noindent for every $\mu\in (0,\bar\mu]$.}

\textrm{Since $u_\mu$ is a critical point of $J_{\mu}$\,, then $%
J_{\mu}^{\prime }(u_\mu)(v)=0$, for any $v \in E_0^\alpha$ and every $\mu\in
(0,\bar\mu]$\,. In particular $J_{\mu}(u_{\mu})(u_{\mu})=0$, that is 
\begin{equation}  \label{rangle}
\Phi^{\prime }(u_{\mu})(u_{\mu})=\mu\int_0^{T} f(u_{\mu}(t))u_{\mu}(t)dt,
\end{equation}
for every $\mu\in (0,\bar\mu]$. }

\textrm{Then, by \eqref{rangle}, it follows that 
\begin{equation*}
\displaystyle 0\leq 2|\cos (\pi \alpha)|
\|u_\mu\|_{\alpha}^2\leq\Phi^{\prime }(u_\mu)(u_\mu)=\mu\int_\Omega
f(u_\mu(t))u_\mu(t)dt,
\end{equation*}
for any $\mu\in (0, \bar\mu]$. Letting $\mu\to 0^+$, by \eqref{bounded}, we
get $\displaystyle
\lim_{\mu\rightarrow 0^+}{\|u_{\mu}\|_{\alpha}}=0$, as claimed. }

\textrm{Finally, we show that the map 
\begin{equation*}
\mu\mapsto J_{\mu}(u_\mu) \,\,\,\, \mbox{is strictly decreasing
in}\,\,\, (0, \mu^{\star}).
\end{equation*}
For this we observe that for any $u\in E_0^{\alpha}$, one has 
\begin{equation}  \label{J11}
J_{\mu}(u)=\mu\left(\frac{\Phi(u)}{\mu}-\Psi(u)\right).
\end{equation}
Now, let us fix $0<\mu_1<\mu_2\leq\bar\mu<\mu^{\star}$ and let $u_{\mu_i}$
be the global minimum of the functional~$J_{\mu_i}$ restricted to $\Phi\big(%
(-\infty, r_{\bar\mu})\big)$ for $i=1,2$. Also, let 
\begin{equation*}
m_{\mu_i}:=\left(\frac{\Phi(u_{\mu_i})}{\mu_i}-\Psi(u_{\mu_i})\right)=%
\inf_{v\in \Phi^{-1}\big((-\infty, r_{\bar\mu})\big)}\left(\frac{\Phi(v)}{%
\mu_i}-\Psi(v)\right),
\end{equation*}
for every $i=1,2$.}

\textrm{Clearly, \eqref{mapnegative} together \eqref{J11} and the positivity
of $\mu$ imply that 
\begin{equation}  \label{mnegativo}
m_{\mu_i}<0,\,\,\,\,\,\, \mbox{for}\,\,\,\, i=1,2.
\end{equation}
Moreover, 
\begin{equation}  \label{m1m2}
m_{\mu_2}\leq m_{\mu_1},
\end{equation}
thanks to the fact that $0<\mu_1<\mu_2$. Then, by \eqref{J11}--\eqref{m1m2}
and again by the fact that $0<\mu_1<\mu_2$, we get that 
\begin{equation*}
J_{\mu_2}(u_{\mu_2})=\mu_2m_{\mu_2}\leq \mu_2m_{\mu_1}<\mu_1m_{\mu_1}=
J_{\mu_1}(u_{\mu_1}),
\end{equation*}
so that the map $\mu\mapsto J_{\mu}(u_{\mu})$ is strictly decreasing in $%
(0,\bar\mu)$.}

\textrm{The arbitrariness of $\bar\mu<\mu^{\star}$ shows that $\mu\mapsto
J_{\mu}(u_{\mu})$ is strictly decreasing in $(0,\mu^\star)$. }
\end{remark}

We would like to note that Theorem~\ref{secondo} is a bifurcation result,
since $\mu=0$ is a bifurcation point for problem $(F_f^{\mu})$, in the sense
that the pair $(0,0)$ belongs to the closure of the set 
\begin{equation*}
\big\{(u_\mu,\mu)\in E_0^{\alpha}\times (0,+\infty) : u_\mu\,\,\,%
\mbox{is a non-trivial weak solution
of $(F_f^{\mu})$}\big\}
\end{equation*}
in $E_0^{\alpha}\times\mathbb{R}$\,.

Indeed, by Theorem~\ref{secondo} we have that 
\begin{equation*}
\|u_\mu\|_{\alpha}\to 0 \,\,\,\,\mbox{as}\,\,\, \mu\to 0^+\,.
\end{equation*}
Hence, there exist two sequences $\big\{u_j\big\}$ in $E_0^{\alpha}$ and $%
\big\{\mu_j\big\}$ in $\mathbb{R}^+$ (here $u_j:=u_{\mu_j}$) such that 
\begin{equation*}
\mu_j\to 0^+ \,\,\,\mbox{and}\,\,\, \|u_j\|_{\alpha}\to 0,
\end{equation*}
as $j\to +\infty$\,.

\bigskip

Moreover, we would like to stress that for any $\mu_1, \mu_2\in (0,
\mu^{\star})$, with $\mu_1\not=\mu_2$, the solutions $u_{\mu_1}$ and $%
u_{\mu_2}$ given by Theorem~\ref{secondo} are different, thanks to the fact
that the map 
\begin{equation*}
(0, \mu^{\star})\ni\mu\mapsto J_{\mu}(u_\mu)
\end{equation*}
is strictly decreasing.

\begin{remark}
\label{Gale}\textrm{{We just observe, for completeness, that Theorem \ref%
{generalkernel0f} remains valid for equations like these 
\begin{equation*}
\frac{d}{dt} \Big({}_0 D_t^{\alpha-1}({}_0^c D_t^{\alpha} u(t)) - {}_t
D_T^{\alpha-1}({}_t^c D_T^{\alpha} u(t))\Big) + f(t,u(t)) = 0, \,\text{a.e. }
t \in [0, T]
\end{equation*}
where $f:\Omega\times\mathbb{R}\rightarrow \mathbb{R}$ is an $L^1$-Carath%
\'{e}odory function. In this case condition $(S_G)$ assume the form 
\begin{equation*}
\sup_{\gamma>0}\frac{\gamma^2}{\displaystyle \int_0^T\max_{|\xi|\leq
\gamma}F(t,\xi)dt}>\frac{\kappa_\alpha}{T},\eqno{(S_G^{\star})}
\end{equation*}
where we set $\displaystyle F(t,\xi):=\int_0^\xi f(t,s)ds$, for every $%
(t,\xi)\in [0,T]\times \mathbb{R}$. Further, the conclusions of Theorem \ref%
{secondo} are still true if we take 
\begin{equation*}
\mu\in\left(0,\frac{T}{\kappa_\alpha}\left(\sup_{\gamma>0}\frac{\gamma^2}{%
\displaystyle \int_0^T\max_{|\xi|\leq \gamma}F(t,\xi)dt}\right)\right).
\end{equation*}
Finally, if $f(t,0)\equiv 0$, we can assume, in order to obtain a
non-trivial solution, that there are a non-empty open set $D\subseteq (0,T)$
and $B\subset D$ of positive Lebesgue measure such that 
\begin{equation*}
\limsup_{\xi\rightarrow 0^+}\frac{\displaystyle\mathop {\rm ess\,inf}_{t\in
B}F(t,\xi)}{\xi^2}=+\infty,
\end{equation*}
and 
\begin{equation*}
\liminf_{\xi\rightarrow 0^+}\frac{\displaystyle\mathop {\rm ess\,inf}_{t\in
D}F(t,\xi)}{\xi^2}>-\infty.
\end{equation*}
} }
\end{remark}

In conclusion, we consider a direct application of Theorem \ref{secondo} and
Remark \ref{atzero}.

\begin{example}
\label{example}\textrm{{Let $\alpha \in (1/2, 1]$ and consider the following
parametric problem $($namely $(F_g^{\mu})$$)$$:$ 
\begin{equation*}
\begin{gathered} \frac{d}{dt} \Big({}_0 D_t^{\alpha-1}({}_0^c D_t^{\alpha}
u(t)) - {}_t D_T^{\alpha-1}({}_t^c D_T^{\alpha} u(t))\Big) + \mu g(u(t)) =
0, \,\text{a.e. } t \in [0, T] \\ u(0) = u(T) = 0, \end{gathered}
\end{equation*}
where the non-linearity $g$ has the form 
\begin{equation}  \label{particularautonomo}
g(u):=\left\{ 
\begin{array}{ll}
u^{r-1}+u^{s-1} & {\mbox{ if }} u\geq 0 \\ 
0 & {\mbox{ otherwise,}}%
\end{array}
\right.
\end{equation}
and in which $1<r<2<s$.}}

\textrm{\noindent Owing to Theorem \ref{secondo} and taking into account
Remark \ref{atzero} problem $(F_g^{\mu})$ admits at least one non-trivial
solution in $E^\alpha_0$ provided 
\begin{equation*}
0<\mu<\frac{rs\bar\gamma^{2-r}}{\kappa_\alpha(s+r\bar\gamma^{s-r})},
\end{equation*}
\noindent where 
\begin{equation*}
\bar\gamma:=\left(\frac{s(2-r)}{r(s-2)}\right)^{1/(s-r)}.
\end{equation*}
Moreover, from Remark \ref{behavious}, one also has 
\begin{equation*}
\lim_{\mu\rightarrow 0^+}\int_0^T |_0^c D_t^{\alpha} u_\mu(t)|^2 dt=0,
\end{equation*}
and the function 
\begin{equation*}
\mu\mapsto - \int_0^T {}_0^c D_t^{\alpha} u_\mu(t) \cdot {}_t^c D_T^{\alpha}
u_\mu(t) dt-\mu \int_0^T \left(\int_0^{u_\mu(t)}f(s)ds\right)dt,
\end{equation*}
is negative and strictly decreasing in $\left(0,\displaystyle\frac{%
rs\bar\gamma^{2-r}}{\kappa_\alpha(s+r\bar\gamma^{s-r})}\right)$. }
\end{example}

\begin{remark}
\label{conclusion} \label{unb2}\textrm{{\ We want to point out that the
energy functional $J_\mu $ related to problem $(F_g^{\mu})$ is not coercive.
Indeed, fix $u\in E_0^\alpha\setminus\{0\}$ and let $\tau\in\mathbb{R}$. We
have 
\begin{equation*}
\begin{aligned} J_{\mu}(\tau u) & = \Phi(\tau u)-\mu\int_0^{T}
\left(\int_0^{\tau u(t)}g(s)ds\right)\,dt\\ & \leq \frac{\tau^2}{|\cos (\pi
\alpha)|} \|u\|_{\alpha}^2-\mu \frac{\tau^r}{r}\|u\|_{L^r}^r-\mu
\frac{\tau^s}{s}\|u\|_{L^s}^s \to -\infty \end{aligned}
\end{equation*}
as $\tau\to +\infty$, bearing in mind that $r<2<s$.} }
\end{remark}

Besides the papers \cite{a1,a2,b1,b2} and \cite{b4,l1,w1,w2,z1} we cite the
monographs \cite{h1,m2,p1} as general references on the subject treated in
this paper. See also \cite{sv, svmountain, svlinking, servadeivaldinociBN}.

\medskip \indent {\bf Acknowledgements.} The authors are grateful to the
referee for the careful analysis of this paper and for constructive remarks.
The paper is realized with the auspices of the GNAMPA Project 2013 entitled: 
\textit{Problemi non-locali di tipo Laplaciano frazionario}.

\end{document}